\documentclass[12pt]{article}
\usepackage{amsmath,amsfonts,amssymb,amsthm,bbm,color}
\usepackage[active]{srcltx}
\usepackage[utf8]{inputenc}

\DeclareMathOperator{\interior}{int}
\DeclareMathOperator{\ext}{ext}
\DeclareMathOperator{\rank}{rank}
\DeclareMathOperator{\dist}{dist}
\DeclareMathOperator{\cone}{cone}
\DeclareMathOperator{\rec}{rec}
\DeclareMathOperator{\conv}{conv}
\DeclareMathOperator{\argmin}{argmin}

\DeclareMathAlphabet{\mymathbb}{U}{bbold}{m}{n}

\begin{document}
\newtheorem{oberklasse}{OberKlasse}
\newtheorem{lemma}[oberklasse]{Lemma}
\newtheorem{proposition}[oberklasse]{Proposition}
\newtheorem{theorem}[oberklasse]{Theorem}
\newtheorem{remark}[oberklasse]{Remark}
\newtheorem{corollary}[oberklasse]{Corollary}
\newtheorem{definition}[oberklasse]{Definition}

\newcommand{\R}{\mathbbm{R}}
\newcommand{\N}{\mathbbm{N}}
\newcommand{\Z}{\mathbbm{Z}}
\newcommand{\mc}{\mathcal}
\newcommand{\eps}{\varepsilon}
\renewcommand{\phi}{\varphi}
\newcommand{\mynote}[1]{{\begin{center}\textcolor{red}{\textsc{#1}}\end{center}}}

\allowdisplaybreaks[1] 

\title{A Galerkin approach to optimization in the space of convex and compact subsets of $\R^d$}
\author{Janosch Rieger\footnote{School of Mathematical Sciences,
Monash University, Victoria 3800, Australia;
telephone: +61(3)99020579; email: janosch.rieger@monash.edu.}}
\date{\today}
\maketitle

\begin{abstract}
The aim of this paper is to establish a theory of Galerkin approximations 
to the space of convex and compact subsets of $\R^d$ with favorable
properties, both from a theoretical and from a computational perspective.
These Galerkin spaces are first explored in depth and then used to solve 
optimization problems in the space of convex and compact subsets of $\R^d$
approximately.
\end{abstract}

{\small\noindent{\bf Keywords:}
Set optimization, Galerkin approximation, convex sets, polytopes.}

\section{Introduction}

In the context of partial differential equations, Galerkin approximations 
are the conceptual link between their analysis and modern numerical methods 
for their solution.
Once the regularity of a solution is established, finite-dimensional 
Galerkin approximations to the corresponding function space are designed,
and efficient numerical methods compute approximate solutions 
to the discretized problems.
This simple, but very powerful idea has been driving the evolution
of major parts of applied mathematics for more than half a century.

\medskip

The aim of this article is to create an analog of Galerkin approximations
for problems posed in the space $\mc{K}_c(\R^d)$ of all nonempty, convex 
and compact subsets of $\R^d$.
Spaces of polytopes with prescribed facet orientations, which have been
described, e.g., in \cite{Alexandrov:05}, are natural candidates
for this purpose.

Our analysis in Section \ref{polsec} shows that these spaces are useful
objects, both from a theoretical as well as from a computational perspective.
They possess a natural system of coordinates, which can be characterized 
as the set of all vectors in $\R^N$ satisfying a linear inequality,
and they have approximation properties, which are very similar to those of
finite element spaces.

In Section \ref{optimsec}, we demonstrate that there exist nested Galerkin
sequences, which are the equivalents of nested conforming finite element schemes
in the world of sets, and that minimizers of auxiliary optimization problems 
posed in the Galerkin spaces approximate minimizers of optimization problems  
posed in $\mc{K}_c(\R^d)$.
As interior point methods are by far the most suitable class of algorithms
for the solution of the resulting finite-dimensional optimization problems,
we devote a substantial part of Section \ref{polsec} to the classification
and geometry of polytopes corresponding to interior and boundary points 
of our coordinate spaces.

\medskip

The contents of this paper found first applications in \cite{Ernst}
and in \cite{Harrach:19}, where the infinite time reachable sets of strictly
stable linear control systems and convex source supports of elliptic impedance
tomography problems are computed as solutions to optimization problems 
in Galerkin polytope spaces.
All results provided in this manuscript have explicitly or implicitly been
exploited in these applications.

\section{The space $\mc{G}_A$} \label{polsec}

In this section, we characterize and analyze the space $\mc{G}_A$ 
of all polyhedra with prescribed facet orientations, which are encoded 
in terms of a matrix $A$ containing the corresponding outer normals.
After collecting some preliminaries in Section \ref{prelim1} and introducing
a natural space $\mc{C}_A$ of coordinates in Section \ref{coor1}, we investigate
in Section \ref{boundedsec}, under which circumstances the spaces $\mc{G}_A$
consist of bounded polytopes.
In Section \ref{charsec}, we characterize $\mc{C}_A$ in terms of a system 
of linear inequalities, and after collecting some information on vertices 
of certain polyhedra in Section \ref{verticessec}, we eliminate some 
redundancies in the characterizing system of inequalities in Section
\ref{eliminatesec}.
In Section \ref{intptsec} we discuss the geometry of polyhedra corresponding
to coordinate vectors in the interior and in different parts of the boundary 
of $\mc{C}_A$, and in Section \ref{approxsec}, we quantify how well the space 
$\mc{G}_A$ approximates $\mc{K}_c(\R^d)$.

\subsection{Setting and preliminaries} \label{prelim1}

We denote the unit sphere in $\R^d$ by $S^{d-1}:=\{x\in\R^d:\|x\|_2=1\}$.
Throughout this section, we will assume that $A\in\R^{N\times d}$ is a fixed matrix which consists of pairwise distinct rows $a_1^T,\ldots,a_N^T$ 
satisfying $a_i\in S^{d-1}$ for $i=1,\ldots,N$.
Denoting $\R_+:=\{s\in\R:s\ge 0\}$, for every $b\in\R^N$ and $c\in\R^d$, 
we define
\begin{align*}
Q_{A,b}:=\{x\in\R^d: Ax\le b\},\quad Q_{A,c}^*:=\{p\in\R^N_+:A^Tp=c\},
\end{align*}
and for $x\in\R^d$ and $p\in\R^N_+$, we set
\[I_x^b:=\{i:a_i^Tx=b_i\},\quad I_p^*:=\{i:p_i>0\}.\]
We define a space of polyhedra by setting 
\[\mc{G}_A:=\{Q_{A,b}:b\in\R^N\}\setminus\{\emptyset\}.\]
The vectors $b\in\R^N$ provide a natural system of coordinates on $\mc{G}_A$.
We denote $\mymathbb{0}=(0,\ldots,0)^T\in\R^N$ and 
$\mathbbm{1}=(1,\ldots,1)^T\in\R^N$.
For any closed convex set $C\subset\R^d$, we denote its set of extremal points 
by $\ext(C)$.

\medskip

Let $\mc{K}_c(\R^d)$ denote the space of all nonempty, convex and compact 
subsets of $\R^d$, and consider the Hausdorff semi-distance 
$\dist:\mc{K}_c(\R^d)\times\mc{K}_c(\R^d)\to\R_+$
and the Hausdorff-distance $\dist_H:\mc{K}_c(\R^d)\times\mc{K}_c(\R^d)\to\R_+$ defined by
\begin{align*}
&\dist(C,\tilde C):=\sup_{c\in C}\inf_{\tilde c\in\tilde C}\|c-\tilde c\|_2,\\
&\dist_H(C,\tilde C):=\max\{\dist(C,\tilde C),\dist(\tilde C,C)\}.
\end{align*}
In addition, we introduce the size and the support function 
\begin{align*}
&\|\cdot\|_2:\mc{K}_c(\R^d)\to\R_+,\quad
\|C\|_2:=\sup_{c\in C}\|c\|_2,\\
&\sigma:\mc{K}_c(\R^d)\times\R^d\to\R,\quad
\sigma_C(x):=\sup_{c\in C}x^Tc.
\end{align*}

The following statements are Corollary 4.5 in \cite{Lauritzen}
and the first theorem in Section 2.5 of \cite{Luenberger}.

\begin{lemma}[Vertices of polyhedra] \label{vertices}
Let $b\in\R^N$ and $c\in\R^d$.
\begin{itemize}
\item [a)] 
A point $x\in Q_{A,b}$ satisfies $x\in\ext(Q_{A,b})$ if and only if 
$x\in Q_{A,b}$ and $\rank(\{a_i:i\in I_x^b\})=d$.
\item [b)] A point $p\in Q_{A,c}^*$ satisfies $p\in\ext(Q_{A,c}^*)$ 
if and only if the vectors $\{a_i:i\in I_p^*\}$ are linearly independent.
\end{itemize}
\end{lemma}

The following facts are Propositions 1.7 and 1.9 in \cite{Ziegler:95}.

\begin{proposition}[Farkas' Lemma] \label{Farkas}
Let $b\in\R^N$, $c\in\R^d$ and $\beta\in\R$.
\begin{itemize}
\item[a)] Either $Q_{A,b}\neq\emptyset$, or there exists $p\in Q_{A,0}^*$ with $b^Tp<0$.
\item[b)] Assume that $Q_{A,b}\neq\emptyset$.
Then the inequality $c^Tx\le\beta$ is true for all $x\in Q_{A,b}$ if and only if
there exists $p\in Q_{A,c}^*$ with $p^Tb\le\beta$.
\end{itemize}
\end{proposition}

Finitely generated convex cones will play a major role in this paper.

\begin{definition}
Given vectors $v_1,\ldots,v_k\in\R^m$, we define the cone generated by these vectors by
$\cone(\{v_1,\ldots,v_k\}):=\{\sum_{j=1}^k\lambda_jv_k:\lambda\in\R^k_+\}$.
\end{definition}

Caratheodory's theorem for cones is taken from 
Theorem 3.14 in \cite{Lauritzen}.

\begin{proposition}[Caratheodory's theorem for cones] \label{Caratheodory}
Let $I\subset\{1,\ldots,N\}$, and let $c\in\cone(\{a_i:i\in I\})$.
Then there exists a subset $J\subset I$ such that $c\in\cone(\{a_i:i\in J\})$
and the vectors $\{a_i:i\in J\}$ are linearly independent.
\end{proposition}

The following version of the strong duality theorem for linear programming 
can be found, e.g., in \cite[Theorem 4.13]{Joswig:Theobald:13}.

\begin{theorem}[Strong duality for LP] \label{strong:duality}
Let $b\in\R^N$ and $c\in\R^d$. 
For the LPs 
\[\max\{c^Tx: x\in Q_{A,b}\}\quad\text{and}\quad\min\{b^Tp:p\in Q_{A,c}^*\},\]
precisely one of the following alternatives holds:
\begin{itemize}
\item [a)] We have $Q_{A,b}\neq\emptyset$ and $Q_{A,c}^*\neq\emptyset$, and
\[\max\{c^Tx: x\in Q_{A,b}\}=\min\{b^Tp:p\in Q_{A,c}^*\}.\]
\item [b)] We have either $Q_{A,b}=\emptyset$ and $\inf\{b^Tp:p\in Q_{A,c}^*\}=-\infty$
or $Q_{A,c}^*=\emptyset$ and $\sup\{c^Tx: x\in Q_{A,b}\}=\infty$. 
\item [c)] We have $Q_{A,b}=\emptyset$ and $Q_{A,c}^*=\emptyset$.
\end{itemize}
\end{theorem}

Hoffman's error bound was first given in the paper \cite{Hoffman:52}.

\begin{theorem}[Hoffman's error bound]\label{Hoffman}
There exists a constant $L_A\ge 0$ such that for all $x\in\R^d$ and all 
$b\in\R^N$ with $Q_{A,b}\neq\emptyset$, we have
\[\dist(x,Q_{A,b})\le L_A\|\max\{0,Ax-b\}\|_\infty,\]
where the maximum is to be interpreted component-wise.
\end{theorem}

The following facts are proved in Lemmas 1.8.12 and 1.8.14 in \cite{Schneider:14}.

\begin{lemma} \label{support:Lipschitz}
For all $x,\tilde x\in\R^d$ and $C,\tilde{C}\in\mc{K}_c(\R^d)$,
we have
\begin{align*}
|\sigma_C(x)-\sigma_{\tilde{C}}(\tilde x)|
\le&\max\{\|C\|_2,\|\tilde{C}\|_2\}\|x-\tilde x\|_2\\
&+\max\{\|x\|_2,\|\tilde{x}\|_2\}\dist_H(C,\tilde{C}),
\end{align*}
as well as
\[\dist_H(C,\tilde C)=\sup_{x\in S^{d-1}}|\sigma_C(x)-\sigma_{\tilde C}(x)|.\]
\end{lemma}

\subsection{Coordinates on $\mc{G}_A$} \label{coor1}

In this section, we check that the mapping 
\[\phi:\mc{C}_A\to\mc{G}_A,\quad \phi(b):=Q_{A,b},\] 
is a bijection between the coordinate space
\[\mc{C}_A:=\{b\in\R^N:\forall\,i\in\{1,\ldots,N\}\ \exists x_i\in Q_{A,b}\ \text{with}\ a_i^Tx_i = b_i\}\]
and the space of polytopes $\mc{G}_A$.
It is immediate that $Q_{A,b}\neq\emptyset$ whenever $b\in\mc{C}_A$.
The following lemmas state that vectors in $\mc{C}_A$ are, in a sense, 
minimal descriptions of polyhedra in $\mc{G}_A$.

\begin{lemma} \label{minimal:lemma}
If $b\in\mc{C}_A$ and $\tilde b\in\R^N$ with $Q_{A,b}=Q_{A,\tilde b}$, 
then $b\le\tilde b$.
\end{lemma}

\begin{proof}
Let $i\in\{1,\ldots,N\}$.
By definition of $\mc{C}_A$, there exists $x_i\in Q_{A,b}$ with $b_i=a_i^Tx_i$,
and since $Q_{A,b}=Q_{A,\tilde{b}}$, we have $a_i^Tx_i\le\tilde b_i$.
\end{proof}

In the next proof, we use Farkas' lemma instead of compactness 
arguments, because we are currently not excluding unbounded polyhedra $Q_{A,b}$.

\begin{lemma} \label{construct:minimal:lemma}
Let $\tilde b\in\R^N$ with $Q_{A,\tilde b}\neq\emptyset$, and define
\begin{equation} \label{loc}
b_i:=\sup\{a_i^Tx:x\in Q_{A,\tilde b}\},\quad i=1,\ldots,N.
\end{equation}
Then $b\in\mc{C}_A$, and we have $b\le\tilde b$ and $Q_{A,b}=Q_{A,\tilde b}$.
\end{lemma}

\begin{proof}
Because of $Q_{A,\tilde b}\neq\emptyset$ and by statement \eqref{loc}, 
we have $-\infty<b_i\le\tilde b_i$ for every $i\in\{1,\ldots,N\}$, 
and hence $b\in\R^N$ and $Q_{A,b}\subset Q_{A,\tilde b}$.
On the other hand, for every $x\in Q_{A,\tilde b}$ and every $i\in\{1,\ldots,N\}$, 
statement \eqref{loc} yields $a_i^Tx\le b_i$, so $Q_{A,\tilde b}\subset Q_{A,b}$.
All in all, we have $Q_{A,b}=Q_{A,\tilde b}\neq\emptyset$.

Now fix $i\in\{1,\ldots,N\}$, and let us check that there exists 
$x_i\in Q_{A,b}$ with $a_i^Tx_i=b_i$.
According to statement \eqref{loc}, for every $\eps>0$, there exists
$x_\eps\in Q_{A,b}$ with $a_i^Tx_\eps>b_i-\eps$.
By Proposition \ref{Farkas}b, this implies that 
\begin{equation} \label{loca}
p^Tb\ge b_i\quad\forall\,p\in Q_{A,a_i}^*.
\end{equation}
Now consider $(p,s)\in\R^N_+\times\R_+$ with $A^Tp-sa_i=0$. 
If $s>0$, then we have $s^{-1}p\in Q_{A,a_i}^*$, so by \eqref{loca}, 
we obtain $b^Tp-b_is\ge 0$.
If $s=0$, then $A^Tp=0$, and since $Q_{A,b}\neq\emptyset$, Proposition \ref{Farkas}a 
yields $b^Tp\ge 0$.
Thus, in both cases we find
\[(b^T,-b_i)\begin{pmatrix}p\\s\end{pmatrix}\ge 0,\]
and according to Proposition \ref{Farkas}a, this implies
\[\{x\in\R^d:Ax\le b,\,-a_i^Tx\le-b_i\}\neq\emptyset,\]
so there exists $x_i\in Q_{A,b}$ with $a_i^Tx_i=b_i$.
\end{proof}

Now we establish the converse of Lemma \ref{minimal:lemma}.

\begin{lemma} \label{minimal:then}
If $b\in\R^N$ and $b\le\tilde b$ for all $\tilde b\in\R^N$ with $Q_{A,b}=Q_{A,\tilde b}$,
then we have $b\in\mc{C}_A$.
\end{lemma}

\begin{proof}
If $Q_{A,b}=\emptyset$, then $\tilde b:=b-\mathbbm{1}<b$ and 
$Q_{A,b-\mathbbm{1}}=\emptyset=Q_{A,b}$, which contradicts the assumption.
Hence $Q_{A,b}\neq\emptyset$, and by Lemma \ref{construct:minimal:lemma}, 
there exists some $\tilde b\in\mc{C}_A$ with $\tilde b\le b$ and $Q_{A,b}=Q_{A,\tilde b}$.
By assumption, we have $b\le\tilde b$, so $b=\tilde b$ and $b\in\mc{C}_A$.
\end{proof}

Let us sum up the preceding discussion.

\begin{proposition} \label{summary}
For $b\in\R^N$, we have $b\in\mc{C}_A$ if and only if $b\le\tilde b$ 
for all $\tilde b\in\R^N$ with $Q_{A,b}=Q_{A,\tilde b}$.
\end{proposition}

\begin{proof}
Combine Lemmas \ref{minimal:lemma} and \ref{minimal:then}.
\end{proof}

Finally, we conclude that $\phi$ is a nice parametrization of $\mc{G}_A$.

\begin{theorem} \label{phi}
The mapping $\phi:\mc{C}_A\to\mc{G}_A$, $\phi(b)=Q_{A,b}$, is a 
homeomorphism between $(\mc{C}_A,\|\cdot\|_\infty)$ and $(\mc{G}_A,\dist_H)$.
The mapping $\phi$ is $L_A$-Lipschitz, and its inverse $\phi^{-1}$ is $1$-Lipschitz.
\end{theorem}

\begin{proof}
By definition of $\mc{C}_A$, it is clear that $\phi(\mc{C}_A)\subset\mc{G}_A$.

\medskip

Let $\tilde b\in\R^N$ with $Q_{A,\tilde b}\in\mc{G}_A$.
By Lemma \ref{construct:minimal:lemma}, there exists $b\in\mc{C}_A$ such that 
$Q_{A,b}=Q_{A,\tilde b}$, so $\phi$ is surjective.
Assume that there exist $b,\tilde b\in\mc{C}_A$ with 
$Q_{A,b}=\phi(b)=\phi(\tilde b)=Q_{A,\tilde b}$.
By Proposition \ref{summary}, we have $b\le\tilde b$ and $\tilde b\le b$, 
so $b=\tilde b$. 
Hence $\phi$ is injective.

\medskip

The Lipschitz property of $\phi$ is a consequence of Hoffman's error bound, 
see Theorem \ref{Hoffman}.
To check the Lipschitz property of the inverse, 
let $b,\tilde b\in\mc{C}_A$ and fix $i\in\{1,\ldots,N\}$. 
By definition of $\mc{C}_A$, there exists $x\in\phi(b)=Q_{A,b}$ with $a_i^Tx=b_i$, 
and there exists $\tilde x\in\phi(\tilde b)=Q_{A,\tilde b}$ with
$\|x-\tilde x\|_2\le\dist(Q_{A,b},Q_{A,\tilde b})$.
But then 
\[b_i-\tilde b_i\le a_i^T(x-\tilde x)\le\|x-\tilde x\|_2
\le\dist(Q_{A,b},Q_{A,\tilde b}).\]
By symmetry, and since the above argument holds for any $i$, we obtain
\[\|b-\tilde b\|_\infty\le\dist(Q_{A,b},Q_{A,\tilde b}).\] 
\end{proof}

\subsection{Spaces of polytopes or unbounded polyhedra} \label{boundedsec}

The recession cone of a convex set describes its behavior at infinity.

\begin{definition}
The recession cone of a closed convex set $C\subset\R^d$ is the set
\[\rec(C):=\{c\in\R^d:x+\lambda c\in C\ \forall\,\lambda\ge 0,\ \forall\,x\in C\}.\]
\end{definition}

We use this notion to prove a theorem of the alternative for the space $\mc{G}_A$.

\begin{theorem} \label{alternative}
Either $\mc{G}_A$ is a collection of unbounded polyhedra, 
or it is a collection of bounded polytopes.
The latter alternative is equivalent with the statement $Q_{A,0}=\{0\}$,
and it is also equivalent with the condition
\[Q_{A,c}^*\neq\emptyset\quad\forall\,c\in\R^d.\]
\end{theorem}

\begin{proof}
Let $Q_{A,b}\in\mc{G}_{A}$ be arbitrary.
According to Theorem 8.4 in \cite{Rockafellar:70}, the set $Q_{A,b}$ is bounded
if and only if $\rec(Q_{A,b})=\{0\}$, and by Proposition 1.12 
from \cite{Ziegler:95}, we have $\rec(Q_{A,b})=Q_{A,0}$.

By definition, we have $0\in\rec(Q_{A,b})$.
By Theorem \ref{strong:duality}, the statement $Q_{A,0}=\{0\}$ implies that
$Q_{A,c}^*\neq\emptyset$ for all $c\in\R^d$.
Conversely, if there exists $c\in Q_{A,0}\setminus\{0\}$, then
$\lambda c\in Q_{A,0}$ for all $\lambda\ge 0$, and 
$\max\{c^Tx:x\in Q_{A,0}\}$ is unbounded.
But then, Theorem \ref{strong:duality} yields $Q_{A,c}^*=\emptyset$.
\end{proof}

The following statement gives some insight into the structure 
of $\mc{G}_A$.

\begin{corollary} \label{cacc}
The space $\mc{G}_A$ is a cone if and only if $\mc{G}_A$ consists of polytopes.
\end{corollary}

\begin{proof}
For any $\lambda>0$ and $b\in\mc{C}_A$, we have 
$\lambda Q_{A,b}=Q_{A,\lambda b}\in\mc{G}_A$.
If $\mc{G}_A$ consists of polytopes, then Theorem \ref{alternative}
yields $0\cdot Q_{A,b}=\{0\}=Q_{A,0}\in\mc{G}_A$, and $\mc{G}_A$ is a cone.
If $\mc{G}_A$ consists of unbounded polyhedra, then $0\cdot Q_{A,b}=\{0\}\notin\mc{G}_A$.
\end{proof}

%
%
%

\subsection{An explicit characterization of $\mc{C}_A$} \label{charsec}

We characterize the set $\mc{C}_A$ in terms of a system of linear inequalities.
To maintain readability, we will denote
\[Q_{A,0}^\diamond=\{q\in\R^N_+: A^Tq=0,\,\mathbbm{1}^Tq=1\}.\]
For the interpretation of Theorem \ref{characterisation},
note that $\ext(Q_{A,0}^\diamond)$ can be empty, while
$e_i\in\ext(Q_{A,a_i}^*)$ for all $i\in\{1,\ldots,N\}$.

\begin{theorem} \label{characterisation}
Let $b\in\R^N$. Then we have $b\in\mc{C}_A$ if and only if 
\begin{subequations}
\begin{align} 
&0\le b^Tp\quad\forall\,p\in Q_{A,0}^\diamond,\label{pre:nonempty}\\
&b_i\le b^Tp\quad\forall\,p\in Q_{A,a_i}^*,\quad\forall\,i\in\{1,\ldots,N\},
\label{pre:touching}
\end{align}
\end{subequations}
which is equivalent with
\begin{subequations}
\begin{align} 
&0\le b^Tp\quad\forall\,p\in\ext(Q_{A,0}^\diamond),\label{nonempty}\\
&b_i\le b^Tp\quad\forall\,p\in\ext(Q_{A,a_i}^*),\quad\forall\,i\in\{1,\ldots,N\}.
\label{touching}
\end{align}
\end{subequations}
\end{theorem}

\begin{remark}
In fact, the conditions \eqref{nonempty} are equivalent with 
$Q_{A,b}\neq\emptyset$, and when $Q_{A,b}\neq\emptyset$, the conditions 
\eqref{touching} guarantee that every inequality $a_i^Tx\le b_i$ is attained.
\end{remark}

\begin{proof}[Proof of Theorem \ref{characterisation}]
By definition, we have $b\in\mc{C}_A$ if and only if
\begin{equation} \label{point1}
\max\{a_i^Tx:x\in Q_{A,b}\}=b_i\quad\forall\,i\in\{1,\ldots,N\}.
\end{equation}
Let us show that \eqref{point1} is equivalent with conditions 
(\ref{pre:nonempty},\,\ref{pre:touching}).

\medskip

Assume that statement \eqref{point1} holds. 
Then we have $Q_{A,b}\neq\emptyset$, so Proposition \ref{Farkas}a yields 
condition \eqref{pre:nonempty}.
Applying Theorem \ref{strong:duality} to statement \eqref{point1} gives
\begin{equation} \label{interm}
\min\{b^Tp:p\in Q_{A,a_i}^*\}=b_i\quad\forall\,i\in\{1,\ldots,N\},
\end{equation}
which implies condition \eqref{pre:touching}.

Conversely, assume that conditions \eqref{pre:nonempty} and \eqref{pre:touching} hold.
Statement \eqref{pre:nonempty} and Proposition \ref{Farkas}a imply $Q_{A,b}\neq\emptyset$.
Statement\eqref{pre:touching} and Theorem \ref{strong:duality} imply
\[b_i\le\min\{b^Tp:p\in Q_{A,a_i}^*\}=\max\{a_i^Tx:x\in Q_{A,b}\}\le b_i
\quad\forall\,i\in\{1,\ldots,N\},\]
so statement \eqref{point1} holds.

\medskip

Since $Q_{A,0}^\diamond$ is a bounded polytope, conditions \eqref{pre:nonempty}
and \eqref{nonempty} are equivalent.
Statement \eqref{pre:touching} clearly implies \eqref{touching}.
Assume that 
\eqref{nonempty} and \eqref{touching} hold.
By Lemma \ref{vertices}b, we have $e_i\in\ext(Q_{A,a_i}^*)$,
so by Theorem 4.24 in \cite{Lauritzen}, we have 
\[Q_{A,a_i}^*=\conv(\ext(Q_{A,a_i}^*))+\rec(Q_{A,a_i}^*)
=\conv(\ext(Q_{A,a_i}^*))+Q_{A,0}^*.\]
Let $p\in Q_{A,a_i}^*$, and let $q\in\conv(\ext(Q_{A,a_i}^*))$ and $r\in Q_{A,0}^*$
with $p=q+r$.
Statement \eqref{touching} gives $b_i\le b^Tq$, and statement \eqref{nonempty} 
implies $0\le b^Tr$, so 
\[b_i\le b^T(q+r)=b^Tp.\]
\end{proof}

We use Theorem \ref{characterisation} to characterize $\mc{C}_A$.
Note that $\ext(Q_{A,0}^\diamond)$ as well as $\ext(Q_{A,a_i})\setminus\{e_i\}$
with $i\in\{1,\ldots,N\}$ can be the empty set.
In this case, the corresponding matrices are to be interpreted as the empty matrix.

\begin{corollary} \label{matrix:representation}
If we enumerate
\begin{align*}
&\ext(Q_{A,0}^\diamond)=\{f^{0,1},\ldots,f^{0,m_0}\},\\
&\ext(Q_{A,a_i})\setminus\{e_i\}=\{f^{i,1},\ldots,f^{i,m_i}\},\quad i\in\{1,\ldots,N\},
\end{align*}
and form the matrix $F:=(F_0,\ldots,F_N)$ with
\begin{align*}
&F_0:=(f^{0,1},\ldots,f^{0,m_0})\in\R^{N\times m_0},\\
&F_i:=(f^{i,1}-e_i,\ldots,f^{i,m_i}-e_i)\in\R^{N\times m_i},\quad  i\in\{1,\ldots,N\},
\end{align*}
then $b\in\mc{C}_A$ is equivalent with $F^Tb\ge 0$.
\end{corollary}

This representation is the key for the practical applicability of the theory
laid out in this paper.
In addition, it has nice theoretical consequences.

\begin{corollary} \label{ca:cone}
The set $\mc{C}_A$ is a closed convex subcone 
of $(\R^N,\|\cdot\|_\infty)$.
\end{corollary}

We can immediately draw the following conclusion.

\begin{corollary} \label{GA:closed}
The metric space $(\mc{G}_A,\dist_H)$ is complete.
\end{corollary}

\begin{proof}
According to Theorem \ref{phi}, the mapping $\phi:\mc{C}_A\to\mc{G}_A$, 
$\phi(b)=Q_{A,b}$, is a bi-Lipschitz homeomorphism between $(\mc{C}_A,\|\cdot\|_\infty)$ 
and $(\mc{G}_A,\dist_H)$, and by Corollary \ref{ca:cone}, the space
$(\mc{C}_A,\|\cdot\|_\infty)$ is complete.
\end{proof}

\subsection{Vertices of $Q_{A,0}^\diamond$ and $Q_{A,c}^*$}
\label{verticessec}

In this section, we gather information on the representation
of vertices that will be used later in the paper.

\begin{lemma} \label{nonempty:vertices}
Let $p\in Q_{A,0}^\diamond$.
Then $p\in\ext(Q_{A,0}^\diamond)$ if and only if 
the vectors $\{(a_i^T,1)^T:i\in I_p^*\}$ are linearly independent.
\end{lemma}

\begin{remark}
An elementary, but lengthy proof shows that the following statement holds:
If $p\in Q_{A,0}^\diamond$, then $p\in\ext(Q_{A,0}^\diamond)$ if and only if 
for every $i_0\in I_p^*$, the vectors $\{a_i:i\in I_p^*\setminus\{i_0\}\}$
are linearly independent.
\end{remark}

\begin{proof}[Proof of Lemma \ref{nonempty:vertices}]
Apply Lemma \ref{vertices}b to 
$Q_{A,0}^\diamond=Q_{(A^T,\mathbbm{1})^T,(0_{\R^d},1)^T}^*$.
\end{proof}

The following statement is an immediate consequence of the above lemma.

\begin{corollary} \label{unique:nonempty}
If $p\in Q_{A,0}^\diamond$ and $\tilde p\in\ext(Q_{A,0}^\diamond)$ satisfy
$I_p^*\subset I_{\tilde p}^*$, then $p=\tilde p$.
\end{corollary}

\begin{proof}
Since $\tilde p\in\ext(Q_{A,0}^\diamond)$, 
the vectors $\{(a_i^T,1)^T:i\in I_{\tilde{p}}^*\}$ are linearly independent
by Lemma \ref{nonempty:vertices} .
The vector $\hat p:=\frac12(p+\tilde p)\in Q_{A,0}^\diamond$ clearly satisfies 
$I_{\hat p}^*=I_{\tilde p}^*$.
Hence Lemma \ref{nonempty:vertices} yields $\hat p\in\ext(Q_{A,0}^\diamond)$,
which forces $p=\tilde p=\hat p$.
\end{proof}

Let us check that $\ext(Q_{A,c}^*)\neq\emptyset$ whenever
$Q_{A,c}^*\neq\emptyset$.

\begin{lemma} \label{has:vertices}
If $c\in\R^d\setminus\{0\}$, then for every $p\in Q_{A,c}^*$, 
there exists some $\tilde{p}\in\ext(Q_{A,c}^*)$ with $I_{\tilde{p}}^*\subset I_p^*$.
\end{lemma}

\begin{proof}
If $p\in Q_{A,c}^*$, then $c\in\cone(\{a_i:i\in I_p^*\})$.
By Proposition \ref{Caratheodory}, there exists $J\subset I_p^*$
such that $c\in\cone(\{a_j:j\in J\})$ and $\{a_j:j\in J\}$ are linearly independent.
Since $c\neq 0$, we have $J\neq\emptyset$, and
there is $\tilde{p}\in Q_{A,c}^*$ with $I_{\tilde{p}}^*\subset J$.
By Lemma \ref{vertices}b, we have $\tilde{p}\in\ext(Q_{A,c}^*)$.
\end{proof}

A statement similar with Corollary \ref{unique:nonempty} 
holds for vertices of $Q_{A,c}^*$.

\begin{lemma} \label{extremal:p}
Let $c\in\R^d$. If $p\in Q_{A,c}^*$ and $\tilde p\in\ext(Q_{A,c}^*)$ satisfy
$I_{p}^*\subset I_{\tilde p}^*$, then we have $p=\tilde p$.
In particular, if $p\in\ext(Q_{A,a_i}^*)\setminus\{e_i\}$, then $p_i=0$.
\end{lemma}

\begin{proof}
By assumption, we can represent 
$\sum_{i\in I_{\tilde p}^*}p_ia_i=c=\sum_{i\in I_{\tilde p}^*}\tilde p_ia_i$.
Since $\tilde p\in\ext(Q_{A,c}^*)$, Lemma \ref{vertices}b yields that
the vectors $\{a_i:i\in I_{\tilde p}^*\}$ are linearly independent,
which forces $p=\tilde p$.
Now let $p\in\ext(Q_{A,a_i}^*)$ with $p_i>0$.
Then the inclusions $e_i\in Q_{A,a_i}^*$ and $I_{e_i}\subset I_p^*$ 
imply $e_i=p$.
\end{proof}

%
%

\subsection{Redundancy in the characterisation of $\mc{C}_A$} \label{eliminatesec}

The system (\ref{nonempty},\,\ref{touching}) is, in general, highly redundant. 
From a practical perspective, redundancies can be eliminated 
in an offline computation.
The results in this section are useful for this elimination process,
and they provide some intuition for the origin of the redundancy.

\begin{theorem}
The system of inequalities \eqref{nonempty} does not contain redundant conditions.
\end{theorem}

\begin{proof}
Let $\ext(Q_{A,0}^\diamond)=\{p^1,\ldots,p^m\}$ with pairwise distinct $p^j\in\R^N_+$.
Assume that for some $k\in\{1,\ldots,m\}$, the condition $(p^k)^Tb\ge 0$ 
is redundant in system \eqref{nonempty}.
By Proposition \ref{Farkas}b, there exists $\lambda\in\R^m_+\setminus\{0\}$ 
with $\lambda_k=0$ and $p^k=\sum_{j=1}^m\lambda_jp^j$. 
Since $p^j\in\R^N_+$ for $j\in\{1,\ldots,m\}$, it follows that
\[I_{p^j}^*\subset I_{p^k}^*\quad\forall j\in\{1,\ldots,m\},\]
so Corollary \ref{unique:nonempty} gives $p^j=p^k$ for all 
$j\in\{1,\ldots,m\}$ with $\lambda_j>0$.
This is a contradiction.
\end{proof}

The way in which the geometry of the matrix $A$ determines the redundancies in
conditions \eqref{touching} is vaguely related to Haar's lemma.
It is currently not clear whether the complete system (\ref{nonempty},\,\ref{touching})
of linear inequalities contains redundancies other than those identified in 
the following theorem.

\begin{theorem} \label{redundancy}
Let $k\in\{1,\ldots,N\}$, and let $p,\tilde p\in\ext(Q_{A,a_k}^*)\setminus\{e_k\}$ with
\begin{equation} \label{cone:contained}
\cone(\{a_i:i\in I_p^*\})\subsetneq\cone(\{a_i:i\in I_{\tilde p}^*\}).
\end{equation}
Then the condition $b_k\le b^T\tilde p$ is redundant in the system 
of inequalities \eqref{touching}.
\end{theorem}

\begin{proof}
First note that the condition $b_k\le b^Tp$ is one of the conditions 
in statement \eqref{touching}, and by \eqref{cone:contained}, 
we have $p\neq\tilde p$.

Again by \eqref{cone:contained}, for all $i\in I_p^*$,
there exist $p^i\in Q_{A,a_i}^*$ with $I_{p^i}^*\subset I_{\tilde p}^*$.
Lemma \ref{vertices}b and $\tilde p\in\ext(Q_{A,a_k}^*)$
imply linear independence of the vectors 
$\{a_i:i\in I_{\tilde p}^*\}$. 
Lemma \ref{vertices}b and linear independence of $\{a_i:i\in I_{p^i}^*\}$ 
imply $p^i\in\ext(Q_{A,a_i}^*)$, so the conditions
\begin{equation} \label{local:3}
b_i\le b^Tp^i\quad\forall\,i\in I_p^*,
\end{equation}
occur in the system of inequalities \eqref{touching} as well.
By Lemma \ref{extremal:p} and since $p\neq e_k$, we have $p_k=0$. 
Hence $a_i\neq a_k$ for all $i\in I_p^*$, which implies
$A^Tp^i=a_i\neq a_k=A^T\tilde{p}$ for all $i\in I_p^*$, so that
$p^i\neq\tilde p$ for all $i\in I_p^*$.

Now we show that the condition $b_k\le b^T\tilde p$ is a consequence 
of the inequalities $b_k\le b^Tp$ and \eqref{local:3}.
We compute
\[\sum_{j\in I_{\tilde p}^*}\tilde p_ja_j
=a_k
=\sum_{i\in I_p^*}p_ia_i
=\sum_{i\in I_p^*}p_i\sum_{j\in I_{\tilde p}^*}p^i_ja_j
=\sum_{j\in I_{\tilde p}^*}(\sum_{i\in I_p^*}p_ip^i_j)a_j,\]
and since $\{a_j:j\in I_{\tilde p}\}$ are linearly independent, it follows that
\begin{equation} \label{p:identity}
\tilde p_j=\sum_{i\in I_p^*}p_ip^i_j\quad\forall\,j\in I_{\tilde p}^*.
\end{equation}
Using \eqref{local:3} and \eqref{p:identity}, we arrive at the estimate
\begin{align*}
b_k\le b^Tp
=\sum_{i\in I_p^*}b_ip_i
\le\sum_{i\in I_p^*}p_i\sum_{j\in I_{\tilde p}^*}b_jp^i_j
=\sum_{j\in I_{\tilde p}^*}b_j\sum_{i\in I_p^*}p_ip^i_j
=\sum_{j\in I_{\tilde p}^*}b_j\tilde p_j
=b^T\tilde p,
\end{align*}
which proves that the condition $b_k\le b^T\tilde p$ is indeed redundant.
\end{proof}

The following immediate consequence of Theorem \ref{redundancy} explains 
the small number of irredundant constraints in \eqref{touching} when $d=2$. 

\begin{corollary} \label{red2d}
Let $t_1<\ldots<t_N\in[0,2\pi)$, and let $a_i^T=(\sin t,\cos t)$.
Then every condition $b_i\le p^Tb$ with $p\in\ext(Q_{A,a_i}^*)$
is redundant in \eqref{touching} unless 
\[I_p^*=\begin{cases}
\{N,2\},&i=1,\\
\{i-1,i+1\},&i\in\{2,\ldots,N-1\},\\
\{N-1,1\},&i=N.
\end{cases}\]
\end{corollary}

\subsection{Interior and boundary points of $\mc{C}_A$} \label{intptsec}

In this section, we will characterize interior and boundary points of $\mc{C}_A$,
which is essential for the use of interior-point methods
for optimization on $\mc{G}_A$.
For any $b\in\mc{C}_A$ and $I\subset\{1,\ldots,N\}$, we define
affine subspaces and facets by
\[H(A,b,I):=\{x\in\R^d:a_i^Tx=b_i,\,i\in I\},\quad
Q_{A,b}^{I}:=Q_{A,b}\cap H(A,b,I).\]

The set $\mc{C}_A$ has nonempty interior, which can be characterized easily
in the setting of Corollary \ref{matrix:representation}.

\begin{theorem} \label{interior:bigger}
The topological interior $\interior(\mc{C}_A)$ of $\mc{C}_A$ in $\R^N$ 
coincides with the set $\{b\in\R^N:F^Tb>0\}$, 
and we have $\mathbbm{1}\in\interior(\mc{C}_A)$.
\end{theorem}

\begin{proof}
By Corollary \ref{matrix:representation}, we have $\mc{C}_A=\{b\in\R^N:F^Tb\ge 0\}$.
Since $F$ does not contain any zero columns, it follows from elementary arguments that 
\[\interior(\mc{C}_A)=\interior(\{b\in\R^N:F^Tb\ge 0\})=\{b\in\R^N:F^Tb>0\}.\]
Let us check that $F^T\mathbbm{1}>0$.
By definition, any vector $f^{0,k}\in\ext(Q_{A,0}^\diamond)$
satisfies $\mathbbm{1}^Tf^{0,k}>0$.
By Lemma \ref{extremal:p}, a vector $f^{i,k}\in\ext(Q_{A,a_i}^*)\setminus\{e_i\}$ 
satisfies $f^{i,k}_i=0$, and we have $f^{i,k}\neq 0$.
Since $a_i^Ta_j<1$ for $i\neq j$, we find
\[\mathbbm{1}^Tf^{i,k}
=\sum_{j=1}^Nf^{i,k}_j
>\sum_{j=1}^Nf^{i,k}_ja_i^Ta_j
=a_i^T\sum_{j=1}^Nf^{i,k}_ja_j
=a_i^Ta_i
=1,\]
which shows that
$(f^{i,k}-e_i)^T\mathbbm{1}>0$, as desired.
\end{proof}

Let us take a closer look at the boundary of $\mc{C}_A$.
When one of the inequalities in \eqref{nonempty} is an equality, then $Q_{A,b}$
is flat.

\begin{proposition} \label{almostgone}
Let $b\in\mc{C}_A$, and let $p\in\ext(Q_{A,0}^\diamond)$. 
Then $b^Tp=0$ holds if and only if $Q_{A,b}\subset Q_{A,b}^{I_p^*}$.
\end{proposition}

\begin{proof}
Let $b^Tp=0$. For any $i\in I_p^*$ and $x\in Q_{A,b}$, we compute
\begin{align*}
-p_ia_i^Tx & = p^TAx - p_ia_i^Tx = \sum_{j\in I_p^*}p_ja_j^Tx - p_ia_i^Tx\\
&= \sum_{j\in I_p^*\setminus\{i\}}p_ja_j^Tx
\le \sum_{j\in I_p^*\setminus\{i\}}p_jb_j = p^Tb-p_ib_i=-p_ib_i,
\end{align*}
which, after division by $-p_i$, gives $a_i^Tx\ge b_i$ and hence $a_i^Tx=b_i$ 
for all $i\in I_p^*$.

Conversely, assume that $Q_{A,b}\subset Q_{A,b}^{I_p^*}$ holds.
Since $b\in\mc{C}_A$, there exists $x\in Q_{A,b}$, and we obtain
\[p^Tb=\sum_{i\in I_p^*}p_ib_i=\sum_{i\in I_p^*}p_ia_i^Tx=(A^Tp)^Tx=0.\]
\end{proof}

In terms of dimension, this means the following.

\begin{proposition} \label{flat}
Let $b\in\mc{C}_A$.
Then $\dim(Q_{A,b})\le d-1$ if and only if there exists $p\in\ext(Q_{A,0}^\diamond)$
with $b^Tp=0$.
\end{proposition}

\begin{proof}
Assume that there exists $p\in\ext(Q_{A,0}^\diamond)$ with $b^Tp=0$.
Since $\mathbbm{1}^Tp=1$, we have $\#I_p^*>0$.
Now Proposition \ref{almostgone} yields $\dim(Q_{A,b})\le\dim(Q_{A,b}^{I_p^*})\le d-1$.

Conversely, assume that $\dim(Q_{A,b})\le d-1$.
Then there exist $\alpha\in\R$ and $c\in\R^d\setminus\{0\}$ such that $c^Tx=\alpha$ 
for all $x\in Q_{A,b}$.
By assumption, we have $Q_{A,b}\neq\emptyset$, and
by Proposition \ref{Farkas}b applied to the inequalities $c^Tx\le\alpha$ 
and $(-c)^Tx\le-\alpha$, we conclude that there exist $q\in Q_{A,c}^*$ with $b^Tq\le\alpha$
and $\tilde q\in Q_{A,-c}^*$ with $b^T\tilde q\le-\alpha$.
Then 
\[\hat p:=\frac{q+\tilde q}{\mathbbm{1}^Tq+\mathbbm{1}^T\tilde q}\in Q_{A,0}^\diamond,\quad
b^T\hat p=\frac{b^Tq+b^T\tilde q}{\mathbbm{1}^Tq+\mathbbm{1}^T\tilde q}\le 0.\]
Since $Q_{A,0}^\diamond$ is a bounded polytope, this implies that there
exists $p\in\ext(Q_{A,0}^\diamond)$ with $b^Tp\le 0$.
By Theorem \ref{characterisation}, we have $b^Tp=0$.
\end{proof}

If one of the inequalities in \eqref{touching} is attained,
the corresponding facet is degenerated.

\begin{proposition} \label{degenerate:vertex}
Let $b\in\mc{C}_A$, and let $p\in\ext(Q_{A,a_k}^*)\setminus\{e_k\}$.
Then $k\notin I_p^*$, and $b^Tp=b_k$ holds if and only if 
$Q_{A,b}^k\subset Q_{A,b}^{I_p^*}$.
\end{proposition}

\begin{proof}
We have $k\notin I_p^*$ by Lemma \ref{extremal:p}.

Assume that $p^Tb=b_k$ holds.
Then for any $i\in I_p^*$ and $x\in Q_{A,b}^k$, we get
\begin{align*}
-p_ia_i^Tx &= p^TAx-a_k^Tx-p_ia_i^Tx 
= \sum_{j\in I_p^*}p_ja_j^Tx-b_k-p_ia_i^Tx\\
&= \sum_{j\in I_p^*\setminus\{i\}}p_ja_j^Tx-b_k 
\le \sum_{j\in I_p^*\setminus\{i\}}p_jb_j-b_k 
=p^Tb-p_ib_i-b_k=-p_ib_i,
\end{align*}
so  $a_i^Tx\ge b_i$.
Conversely, let $Q_{A,b}^k\subset Q_{A,b}^{I_p^*}$.
Since $b\in\mc{C}_A$, there exists some $x\in Q_{A,b}^k$, and we find
\[p^Tb=\sum_{i\in I_p^*}p_ib_i=\sum_{i\in I_p^*}p_ia_i^Tx=(A^Tp)^Tx=a_k^Tx=b_k.\]
\end{proof}

Alternatively, we can describe this situation from an algebraic perspective:
If one of the inequalities in \eqref{touching} is an equality, then the
corresponding condition $a_k^Tx\le b_k$ is redundant in the definition of $Q_{A,b}$.

\begin{proposition}
Let $b\in\mc{C}_A$, and let $p\in\ext(Q_{A,a_k}^*)\setminus\{e_k\}$.
Then $k\notin I_p^*$, and $b^Tp=b_k$ holds if and only if  
$a_i^Tx\le b_i$ for all $i\in I_p^*$ implies $a_k^Tx\le b_k$.
\end{proposition}

\begin{proof}
Again, we have $k\notin I_p$ by Lemma \ref{extremal:p}.

If we assume that $b^Tp=b_k$ holds and that $a_i^Tx\le b_i$ for all $i\in I_p^*$, 
then $a_k^Tx\le b_k$ follows directly from $A^Tp=a_k$ and Proposition \ref{Farkas}b.

Conversely, assume that $a_i^Tx\le b_i$ for all $i\in I_p^*$ implies $a_k^Tx\le b_k$.
By Proposition \ref{Farkas}b, there exists $\tilde p\in Q_{A,a_k}^*$ with 
$I_{\tilde p}^*\subset I_p^*$ and $b^T\tilde p\le b_k$.
Corollary \ref{unique:nonempty} yields $\tilde p=p$, so $b^Tp\le b_k$,  
and $b\in\mc{C}_A$ implies $b^Tp\ge b_k$ by Theorem \ref{characterisation}.
\end{proof}

The correspondence between dimensionality and algebraic inequalities 
is more complicated for facets $Q_{A,b}^k$ than for the entire polyhedron $Q_{A,b}$.

\begin{corollary} \label{noface}
Let $b\in\mc{C}_A$, and let $p\in\ext(Q_{A,a_k}^*)\setminus\{e_k\}$.
Then $b^Tp=b_k$ implies $\dim(Q_{A,b}^k)\le d-2$.
\end{corollary}

\begin{proof}
By Proposition \ref{degenerate:vertex}, the identity $b^Tp=b_k$ implies 
$Q_{A,b}^k\subset Q_{A,b}^{I_p^*}$.
Since $p\in\ext(Q_{A,a_k}^*)\setminus\{e_k\}$,
Lemma \ref{extremal:p} yields $p_k=0$, so $\#I_p^*\ge 2$.
By Lemma \ref{vertices}b, the vectors $\{a_i:i\in I_p^*\}$ are linearly independent.
Hence we conclude $\dim(Q_{A,b}^k)\le\dim(H(A,b,I_p^*))\le d-2$.
\end{proof}

The following statement is a semi-converse of Corollary \ref{noface}.

\begin{proposition} \label{ecafon}
Let $b\in\mc{C}_A$.
If $\dim(Q_{A,b}^k)\le d-2$, then $\dim(Q_{A,b})\le d-1$ or there exists 
$p\in\ext(Q_{A,a_k}^*)\setminus\{e_k\}$ with $b^Tp=b_k$.
\end{proposition}

\begin{proof}
If $\dim(Q_{A,b}^k)\le d-2$, then there exist $\alpha\in\R$ and $c\in\R^d$ 
with $c\neq 0$ such that $\{a_k,c\}$ are linearly independent and $c^Tx=\alpha$ holds 
for all $x\in Q_{A,b}^k$.
Applying Proposition \ref{Farkas}b to the inequalities 
\[c^Tx\le\alpha\ \forall\,x\in Q_{A,b}^k,\quad (-c)^Tx\le-\alpha\ \forall\,x\in Q_{A,b}^k,\]
yields $q,\tilde q\in\R^N_+$ and $t,\tilde t\in\R_+$ with 
\begin{align*}
&A^Tq-ta_k=c,&& b^Tq-tb_k\le\alpha,\\
&A^T\tilde q-\tilde ta_k=-c,&& b^T\tilde q-\tilde tb_k\le-\alpha.
\end{align*}
Since $\{a_k,c\}$ are linearly independent, 
there exist $\ell,\tilde\ell\in\{1,\ldots,N\}\setminus\{k\}$ with $q_\ell>0$ and 
$\tilde{q}_{\tilde\ell}>0$.
Setting $\hat p:=q+\tilde q$ and $s:=t+\tilde t$, we obtain
$\hat{p}\in\R^N_+$, $\hat{p}_\ell,\hat{p}_{\tilde\ell}>0$, $s\ge 0$ and 
\begin{equation} \label{two:ways}
A^T\hat p=sa_k,\quad b^T\hat p\le sb_k.
\end{equation}

\emph{Case 1:} If $\hat{p}_k\ge s$, define $\bar{p}:=\hat{p}-se_k$.
Then $\bar{p}\in\R^N_+\setminus\{0\}$, and from statement \eqref{two:ways} we have
$A^T\bar{p}=0$ and $b^T\bar{p}\le 0$.
In particular, we have $\frac{\bar{p}}{\mathbbm{1}^T\bar{p}}\in Q_{A,0}^\diamond$,
and Theorem \ref{characterisation} yields $b^T\frac{\bar{p}}{\mathbbm{1}^T\bar{p}}=0$.
Since $Q_{A,0}^\diamond$ is a bounded polytope, Proposition \ref{flat}
yields $\dim(Q_{A,b})\le d-1$.

\emph{Case 2:} If $\hat{p}_k<s$, define $\bar{p}:=\hat{p}-\hat{p}_ke_k$.
Then $\bar{p}\in\R^N_+\setminus\{0\}$ and $\bar{p}_k=0$, and statement
\eqref{two:ways} yields
\begin{equation} \label{loco}
A^T\frac{\bar{p}}{s-\hat{p}_k}=a_k,\quad b^T\frac{\bar{p}}{s-\hat{p}_k}\le b_k.
\end{equation}
In particular, we have $\frac{\bar{p}}{s-\hat{p}_k}\in Q_{A,a_k}^*$.
As in the proof of Theorem \ref{characterisation}, we can write 
$\frac{\bar{p}}{s-\hat{p}_k}=v+w$ with $v\in\conv(\ext(Q_{A,a_k}^*))$ and
$w\in Q_{A,0}^*$.
Since $\bar{p}_k=0$ and $v,w\ge 0$, we have $v_k=0$.
Denote $\ext(Q_{A,a_k}^*)\setminus\{e_k\}=\{f^{k,1},\ldots,f^{k,m_k}\}$, and
let $\lambda\in\R^{m_k}_+$ and $\mu\ge 0$ with $\mathbbm{1}^T\lambda+\mu=1$ 
and $v=\sum_{j=1}^{m_k}\lambda_jf^{k,j}+\mu e_k$.
Then $v_k=0$ and $f^{k,j}\ge 0$ for all $j\in\{1,\ldots,m_k\}$ force $\mu=0$, 
so we have 
\begin{equation} \label{ek:does:not:matter}
\sum_{j=1}^{m_k}\lambda_j=1,\quad v=\sum_{j=1}^{m_k}\lambda_jf^{k,j}.
\end{equation}
By Theorem \ref{characterisation} and by statement \eqref{loco}, we have 
\[\sum_{j=1}^{m_k}\lambda_jb^Tf^{k,j}=b^Tv\le b^Tv+b^Tw=b^T(v+w)\le b_k.\]
But Theorem \ref{characterisation} also guarantees $b^Tf^{k,j}\ge b_k$
for all $j\in\{1,\ldots,m_k\}$, which implies $b^Tf^{k,j}=b_k$ 
for every $j\in\{1,\ldots,m_k\}$ with $\lambda_j>0$.
By statement \eqref{ek:does:not:matter}, there exists at least one such $j$,
which completes the proof.
\end{proof}

Now we characterize the polyhedra $Q_{A,b}$, which correspond to interior points
of $\mc{C}_A$.
Recall the definition of the matrix $F$ from Corollary \ref{matrix:representation}.

\begin{theorem} \label{inner:points}
We have $b\in\interior\mc{C}_A$ if and only if 
\begin{equation} \label{int:char}
\dim Q_{A,b}=d,\quad\dim Q_{A,b}^k=d-1\quad\forall\,k\in\{1,\ldots,N\}.
\end{equation} 
\end{theorem}

\begin{proof}
According to Theorem \ref{interior:bigger}, we have $b\in\interior\mc{C}_A$
if and only if $F^Tb>0$.
If $F^Tb>0$, then Propositions \ref{flat} and \ref{ecafon} imply \eqref{int:char}.
Conversely, if condition \eqref{int:char} holds, then Proposition \ref{flat}
and Corollary \ref{noface} imply $F^Tb>0$.
\end{proof}

%

\subsection{Approximation properties} \label{approxsec}

First, we introduce a projector from $\mc{K}_c(\R^d)$ to $\mc{G}_A$.

\begin{proposition} \label{projectorprop}
Let $\phi$ and $L_A$ as in Section \ref{coor1}.
The mapping 
\begin{align*}
P_{\mc{C}_A}:\mc{K}_c(\R^d)\to\mc{C}_A,\quad 
P_{\mc{C}_A}(C):=(\sigma_C(a_1),\ldots,\sigma_C(a_N))
\end{align*}
is well-defined and $1$-Lipschitz from $(K_c(\R^d),\dist_H)$ to $(\R^N,\|\cdot\|_\infty)$
with 
\[P_{\mc{C}_A}(C)\le P_{\mc{C}_A}(\tilde C)\quad\forall\,C,\tilde C\in\mc{K}_c(\R^d)\ 
\text{with}\ C\subset\tilde C,\]
and the mapping
\begin{align*}
P_{\mc{G}_A}:\mc{K}_c(\R^d)\to\mc{G}_A,\quad 
P_{\mc{G}_A}(C):=\phi(P_{\mc{C}_A}(C))
\end{align*}
is an $L_A$-Lipschitz projector 
from $(K_c(\R^d),\dist_H)$ onto $(\mc{G}_A,\dist_H)$ with
\[P_{\mc{G}_A}(C)\subset P_{\mc{G}_A}(\tilde C)\quad\forall\,C,\tilde C\in\mc{K}_c(\R^d)\ 
\text{with}\ C\subset\tilde C.\]
\end{proposition}

\begin{proof}
If $C\in\mc{K}_c(\R^d)$,
then for all $k\in\{1,\ldots,N\}$, there exists $x_k\in C$ with
\[a_k^Tx_k=\sup_{x\in C}a_k^Tx=\sigma_C(a_k),\quad 
a_l^Tx_k\le\sigma_C(a_l)\ \forall\,\ell\in\{1,\ldots,N\}.\]
In particular $x_k\in Q_{A,P_{\mc{C}_A}(C)}$ for all $k\in\{1,\ldots,N\}$,
so $P_{\mc{C}_A}(C)\in\mc{C}_A$, and the mapping $P_{\mc{C}_A}$ is well-defined.

It follows from Lemma \ref{support:Lipschitz} and $\|a_k\|_2=1$ for
$k\in\{1,\ldots,N\}$ that
\[\|P_{\mc{C}_A}(C)-P_{\mc{C}_A}(\tilde{C})\|_\infty\le\dist_H(C,\tilde{C})
\quad\forall\,C,\tilde C\in\mc{K}_c(\R^d).\]
Since $\phi$ is $L_A$-Lipschitz according to Theorem \ref{phi}, so is $P_{\mc{G}_A}$.
By construction,
\[P_{\mc{G}_A}(Q_{A,b})=\phi(P_{\mc{C}_A}(Q_{A,b}))
=\phi(b)=Q_{A,b}\quad\forall\,b\in\mc{C}_A,\]
so $P_{\mc{G}_A}$ is indeed a projector from $\mc{K}_c(\R^d)$ onto $\mc{G}_A$.
\end{proof}

Now we investigate the quality of the approximation of $\mc{K}_c(\R^d)$ by $\mc{G}_A$.
Theorem \ref{oldapprox}, originally proved in \cite{Rieger:11}, provides a measure 
in terms of the metric density
\[\delta_A:=\sup_{c\in S^{d-1}}\min\{\|c-a_k\|_2:k=1,\ldots,N\}\]
of the vectors $\{a_k:k=1,\ldots,N\}$ in the sphere $S^{d-1}$.
The assumption $\delta_A\in(0,1)$ is only restrictive when working
with very coarse spaces $\mc{G}_A$ in high-dimensional ambient spaces $\R^d$.

\begin{theorem} \label{oldapprox}
Let $\mc{G}_A$ be a space of polytopes (see Theorem \ref{alternative}), 
and let $\delta_A\in(0,1)$.
Then for every $C\in\mc{K}_c(\R^d)$, 
we have $C\subset P_{\mc{G}_A}(C)$ and
\[\dist(P_{\mc{G}_A}(C),C)\le\frac{2-\delta_A}{1-\delta_A}\delta_A\|C\|_2.\]
\end{theorem}

\begin{proof}
The inclusion $C\subset P_{\mc{G}_A}(C)$ holds by construction of $P_{\mc{G}_A}$,
and since $\mc{G}_A$ is a space of polytopes, Lemma \ref{support:Lipschitz} 
applied with $\tilde{C}=\{0\}$ yields 
\[\|\sigma_{P_{\mc{G}_A}(C)}\|_\infty=\|P_{\mc{G}_A}(C)\|_2<\infty.\]
Let $x\in P_{\mc{G}_A}(C)$ and $c\in S^{d-1}$.
By assumption, there exists $k\in\{1,\ldots,N\}$ with $\|c-a_k\|_2\le\delta_A$, 
so, again by Lemma \ref{support:Lipschitz}, we have
\begin{align*}
c^Tx 
&=(c-a_k)^Tx+a_k^Tx\\
&\le\|c-a_k\|_2\,\|x\|_2+\|\sigma_C\|_\infty
\le\delta_A\|\sigma_{P_{\mc{G}_A}(C)}\|_\infty+\|\sigma_C\|_\infty.
\end{align*}
Since $x$ and $c$ were arbitrary, we have
\[\|\sigma_{P_{\mc{G}_A}(C)}\|_\infty
\le\delta_A\|\sigma_{P_{\mc{G}_A}(C)}\|_\infty+\|\sigma_C\|_\infty,\]
so that
\begin{equation} \label{inter}
\|\sigma_{P_{\mc{G}_A}(C)}\|_\infty\le\frac{1}{1-\delta_A}\|\sigma_C\|_\infty.
\end{equation}
Again, let $x\in P_{\mc{G}_A}(C)$, $c\in S^{d-1}$ and 
$k\in\{1,\ldots,N\}$ with $\|c-a_k\|_2\le\delta_A$.
Using inequality \eqref{inter}, Lemma \ref{support:Lipschitz} and  
$\sigma_C(a_k)=\sigma_{P_{\mc{G}_A}(C)}(a_k)$, we obtain
\begin{align*}
&c^Tx-\sigma_C(c)
=(c-a_k)^Tx+a_k^Tx-\sigma_C(c)\\
&\le\|c-a_k\|_2\,\|x\|_2+\sigma_C(a_k)-\sigma_C(c)
\le\delta_A\|\sigma_{P_{\mc{G}_A}(C)}\|_\infty+\delta_A\|C\|_2\\
&\le\frac{\delta_A}{1-\delta_A}\|\sigma_C\|_\infty+\delta_A\|C\|_2
=\frac{2-\delta_A}{1-\delta_A}\delta_A\|C\|_2.
\end{align*}
Since $x$ and $c$ were arbitrary, it follows from Lemma \ref{support:Lipschitz} that
\[\dist_H(P_{\mc{G}_A}(C),C)
=\|\sigma_{P_{\mc{G}_A}(C)}-\sigma_C\|_\infty
\le\frac{2-\delta_A}{1-\delta_A}\delta_A\|C\|_2.\]
\end{proof}

While the number $\delta_A$ only measures metric density, the quantity
\[\kappa_A:=\sup_{c\in S^{d-1}}\,
\inf\Big\{\sum_{k\in I_p^*}p_k\|a_k-\frac{c}{\|p\|_1}\|_2:p\in Q_{A,c}^*\Big\}\]
encodes the geometry of the matrix $A$.
Let us first check that it is well-defined when $\mc{G}_A$ is a space of bounded
polytopes and the geometry of $A$ is sufficiently rich.
For the interpretation of the following proposition, 
note that $\lim_{\rho\nearrow 1}\sqrt{(2-2\rho)/\rho}=0$.

\begin{proposition} 
Assume that there exists 
$\rho\in(0,1)$ such that for every $c\in S^{d-1}$, there is $p\in Q_{A,c}^*$ 
with $\min_{i,j\in I_p^*}a_i^Ta_j\ge\rho$.
Then
\[\kappa_A\in[0,\sqrt{\frac{2-2\rho}{\rho}}].\]
\end{proposition}

\begin{proof}
Let $c\in S^{d-1}$. 
By assumption, there exists $p\in\R^N_+$ with $A^Tp=c$ and 
$\min_{i,j\in I_p^*}a_i^Ta_j\ge\rho$, so
\begin{align}
&\rho\|p\|_1^2=\rho\sum_{i,j\in I_p^*}p_ip_j
\le\sum_{i,j\in I_p^*}p_ia_i^Ta_jp_j,\label{ha}\\
&\|p\|_1^2=\sum_{i,j\in I_p^*}p_ip_j\ge\sum_{i,j\in I_p^*}p_ia_i^Ta_jp_j,\label{hi}\\
&1=\|c\|^2_2=\|A^Tp\|_2^2=\sum_{i,j\in I_p^*}p_ia_i^Ta_jp_j,\label{ho}
\end{align}
and combining statements \eqref{hi} and \eqref{ho}, we obtain
\begin{align*}
&\|a_k-\frac{c}{\|p\|_1}\|_2^2
=\|a_k\|_2^2-\frac{2}{\|p\|_1}a_k^Tc+\frac{\|c\|^2}{\|p\|_1^2}\\
&=1-\frac{2}{\|p\|_1}\sum_{i\in I_p^*}p_ia_k^Ta_i+\frac{1}{\|p\|_1^2}
\le 2-2\rho.
\end{align*}
Using this and combining statements \eqref{ha} and \eqref{ho}, we arrive at
\[\sum_{k\in I_p^*}p_k\|a_k-\frac{c}{\|p\|_1}\|_2\le\sqrt{\frac{2-2\rho}{\rho}}.\]
\end{proof}

Now we estimate the quality of the approximation of $\mc{K}_c(\R^d)$ by $\mc{G}_A$.

\begin{theorem} \label{projector}
Let $\mc{G}_A$ be a space of polytopes.
Then for every $C\in\mc{K}_c(\R^d)$, 
we have $C\subset P_{\mc{G}_A}(C)$ and
\[\dist(P_{\mc{G}_A}(C),C)\le\kappa_A\|C\|_2.\]
\end{theorem}

\begin{proof}
The definition of $P_{\mc{G}_A}$ implies $C\subset P_{\mc{G}_A}(C)$.
Let us fix $C\in\mc{K}_c(\R^d)$ and $z\in P_{\mc{G}_A}(C)$.
Then for every $c\in S^{d-1}$ and every $p\in\ext(Q_{A,c}^*)$, 
we obtain, using Lemma \ref{support:Lipschitz}, that
\begin{align*}
&c^Tz-\sigma_C(c)
=\sum_{k=1}^Np_ka_k^Tz-\Big(\sum_{k=1}^N\frac{p_k}{{\|p\|_1}}\Big)\sigma_C(c)
\le\sum_{k=1}^Np_k\sigma_C(a_k)-\sum_{k=1}^Np_k\frac{\sigma_C(c)}{\|p\|_1}\\
&=\sum_{k=1}^Np_k\Big(\sigma_C(a_k)-\sigma_C(\frac{c}{\|p\|_1})\Big)
\le\|C\|_2\sum_{k=1}^Np_k\Big\|a_k-\frac{c}{\|p\|_1}\Big\|_2.
\end{align*}
It follows from $C\subset P_{\mc{G}_A}(C)$, Lemma \ref{support:Lipschitz} 
and the above computation that 
\[\dist(P_{\mc{G}_A}(C),C)\le\sup_{c\in S^{d-1}}|\sigma_{P_{\mc{G}_A}(C)}(c)-\sigma_C(c)|
\le\kappa_A\|C\|_2.\]
\end{proof}

\section{Galerkin optimization in $\mc{K}_c(\R^d)$} \label{optimsec}

In this section, we use the spaces analyzed in Section \ref{polsec}
to solve optimization problems in $\mc{K}_c(\R^d)$ approximately.
After gathering a few preliminaries in Section \ref{prelim2}
we prove a convergence result for an abstract set optimization problem 
and suitable auxiliary problems in Section \ref{abstract}.
In Section \ref{galerseq}, we introduce the concept of Galerkin approximations 
to $\mc{K}_c(\R^d)$, and in Section \ref{concrete}, we show in detail that 
an important class of optimization problems in $\mc{K}_c(\R^d)$ and their 
Galerkin approximations are a special case of the abstract framework discussed 
in Section \ref{abstract}.

\subsection{Preliminaries} \label{prelim2}

All notions of convergence, continuity and compactness are to be understood in 
terms of the Hausdorff distance $\dist_H$.
We equip the space 
of all compact subsets of $(\mc{K}_c(\R^d),\dist_H)$
with the Hausdorff semi-distance and the Hausdorff-distance given by
\begin{align*}
&\mc{D}:2^{\mc{K}_c(\R^d)}\times 2^{\mc{K}_c(\R^d)}\to\R_+,\quad
\mc{D}(\mc{M},\tilde{\mc{M}}):=\sup_{C\in\mc{M}}\inf_{\tilde C\in\tilde{\mc{M}}}\dist_H(\mc{M},\tilde{\mc{M}}),\\
&\mc{D}_H:2^{\mc{K}_c(\R^d)}\times2^{\mc{K}_c(\R^d)}\to\R_+,\
\mc{D}_H(\mc{M},\tilde{\mc{M}}):=\max\{\mc{D}(\mc{M},\tilde{\mc{M}}),\mc{D}(\tilde{\mc{M}},\mc{M})\}.
\end{align*}

Let us fix some notation for the objective function. 

\begin{definition}
Consider a functional $\Phi:\mc{K}_c(\R^d)\to\R$.
\begin{itemize}
\item [a)] The function $\Phi$ is called lower semicontinuous
if for every $C\in\mc{K}_c(\R^d)$ and every sequence 
$(C_k)_{k=0}^\infty\subset\mc{K}_c(\R^d)$ with
$\lim_{k\to\infty}\dist_H(C_k,C)=0$, we have
$\liminf_{k\to\infty}\Phi(C_k)\ge\Phi(C)$.
\item [b)] For every $\beta\in\R$, we denote 
$S(\Phi,\beta):=\{C\in\mc{K}_c(\R^d):\Phi(C)\le\beta\}$.
\end{itemize}
\end{definition}

The following result is Theorem 1.8.7 in \cite{Schneider:14}.

\begin{theorem}[Blaschke selection theorem] \label{Blaschke}
For every $R>0$, the collection $\{C\in\mc{K}_c(\R^d):\|C\|_2\le R\}$
is 
compact.
\end{theorem}

\subsection{An abstract framework} \label{abstract}

The following statements are variations of well-known facts.

\begin{lemma} \label{exargmin}
Let $\Phi:\mc{K}_c(\R^d)\to\R$ be lower semicontinuous, let $\beta\in\R$, 
and let $\mc{M}\subset\mc{K}_c(\R^d)$ be a closed set.
Then the following statements hold.
\begin{itemize}
\item [a)] The set $S(\Phi,\beta)$ is closed.
\item [b)] The set $\argmin_{C\in\mc{M}}\Phi(C)$ is closed.
\item [c)] If $\mc{M}\cap S(\Phi,\beta)$ is nonempty and compact,
then $\argmin_{C\in\mc{M}}\Phi(C)\neq\emptyset$.
\end{itemize}
\end{lemma}

The sets of global minima of suitable auxiliary problems 
converge to the set of global minima of the original
optimization problem.

\begin{theorem} \label{minconvthm}
Let $\mc{M}\subset\mc{K}_c(\R^d)$ be nonempty and compact, 
and let $(\mc{M}_k)_{k=0}^\infty$ be a sequence of nonempty and compact subsets 
$\mc{M}_k\subset\mc{K}_c(\R^d)$ with
\begin{align} \label{setconv}
\lim_{k\to\infty}\mc{D}_H(\mc{M},\mc{M}_k)=0.
\end{align}
Let $\Phi:\mc{K}_c(\R^d)\to\R$ be continuous, let $(\Phi_k)_{k=0}^\infty$
be a sequence of mappings $\Phi_k:\mc{M}_k\to\R$ satisfying
\begin{align} \label{uniapprox}
\lim_{k\to\infty}\sup_{C\in\mc{M}_k}|\Phi(C)-\Phi_k(C)|=0,
\end{align} 
and let $\argmin_{C\in\mc{M}_k}\Phi_k(C)\neq\emptyset$.
Then $\argmin_{C\in\mc{M}}\Phi(C)\neq\emptyset$, and 
\begin{equation} \label{desiredconv}
\lim_{k\to\infty}\mc{D}(\argmin_{C\in\mc{M}_k}\Phi_k(C),
\argmin_{C\in\mc{M}}\Phi(C))=0.
\end{equation}
\end{theorem}

\begin{proof}
Consider a subsequence $\N'\subset\N$ and sets
$C_k^*\in\argmin_{C\in\mc{M}_k}\Phi_k(C)$ for all $k\in\N'$.
By statement \eqref{setconv}, there exists a sequence 
$(C_k)_{k\in\N'}\subset\mc{M}$ with
\[\lim_{\N'\ni k\to\infty}\dist_H(C_k,C_k^*)=0.\]
Since $\mc{M}$ is compact, there exist $C^*\in\mc{M}$ and a subsequence 
$\N''\subset\N'$ with $\lim_{\N''\ni k\to\infty}\dist_H(C_k,C^*)=0$, 
so all in all, we have
\begin{equation} \label{converge}
\lim_{\N''\ni k\to\infty}\dist_H(C_k^*,C^*)=0.
\end{equation}
Continuity of $\Phi$ and statements \eqref{uniapprox} and \eqref{converge} yield
\begin{equation}
\left.\begin{aligned} 
&\lim_{\N''\ni k\to\infty}|\Phi(C^*)-\Phi_k(C_k^*)|\\
&\le\lim_{\N''\ni k\to\infty}|\Phi(C^*)-\Phi(C_k^*)|
+\lim_{\N''\ni k\to\infty}|\Phi(C_k^*)-\Phi_k(C_k^*)|=0.
\end{aligned}\right\}\label{ssvalue}
\end{equation}
Let $C\in\mc{M}$. 
By statement \eqref{setconv}, there exists $(\tilde C_k)_{k\in\N''}$ 
with $\tilde C_k\in\mc{M}_k$ and 
\begin{equation} \label{convlow}
\lim_{\N''\ni k\to\infty}\dist_H(\tilde C_k,C)=0.
\end{equation}
Again, statements \eqref{uniapprox} and \eqref{convlow} yield
\begin{equation}
\left.\begin{aligned} 
&\lim_{\N''\ni k\to\infty}|\Phi(C)-\Phi_k(\tilde C_k)|\\
&\le\lim_{\N''\ni k\to\infty}|\Phi(C)-\Phi(\tilde C_k)|
+\lim_{\N''\ni k\to\infty}|\Phi(\tilde C_k)-\Phi_k(\tilde C_k)|=0,
\end{aligned}\right\}\label{sslvalue}
\end{equation}
and because of statements \eqref{ssvalue} and \eqref{sslvalue}, we have
\[\Phi(C)=\lim_{\N''\ni k\to\infty}\Phi_k(\tilde C_k)
\ge\lim_{\N''\ni k\to\infty}\Phi_k(C_k^*)=\Phi(C^*).\]
All in all, we have $C^*\in\argmin_{C\in\mc{M}}\Phi(C)$.

\medskip

Now assume that statement \eqref{desiredconv} is false.
Then there exist $\eps>0$, a subsequence $\N'\subset\N$ and sets
$C_k^*\in\argmin_{C\in\mc{M}_k}\Phi_k(C)$ for all $k\in\N'$ with
\begin{equation} \label{contradict}
\mc{D}(C_k^*,\argmin_{C\in\mc{M}}\Phi(C))\ge\eps\quad\forall\,k\in\N'.
\end{equation} 
But the first part of the proof shows that there exists 
$C^*\in\argmin_{c\in\mc{M}}\Phi(C)$ such that statement \eqref{converge} holds.
This contradicts statement \eqref{contradict}.
\end{proof}

\subsection{Galerkin sequences} \label{galerseq}

Now we introduce the equivalent to Galerkin schemes from the realm of 
partial differential equations.

\begin{definition} \label{galseq}
A sequence $(A_k)_{k=0}^\infty$ of matrices $A_k\in\R^{N_k\times d}$ with
$N_k\in\N$, $k\in\N$, is called a Galerkin sequence if there exists 
a sequence $(\alpha_k)_{k=0}^\infty\in\R_+$ with $\lim_{k\to\infty}\alpha_k=0$ 
such that 
\[\inf_{\tilde C\in\mc{G}_{A_k}}\dist_H(C,\tilde C)\le\alpha_k\|C\|_2\quad\forall\,C\in\mc{K}_c(\R^d).\]
If, in addition, $A_k$ is a submatrix of $A_{k+1}$ for all $k\in\N$,
then we call $(A_k)_{k=0}^\infty$ a nested Galerkin sequence.
\end{definition}

Let us draw some immediate conclusions from Definition \ref{galseq}.

\begin{lemma} \label{immediate}
If $(A_k)_{k=0}^\infty$ is a Galerkin sequence,
the following statements hold.
\begin{itemize}
\item [a)] The spaces $\mc{G}_{A_k}$  consist of polytopes.
\item [b)] If $(A_k)_{k=0}^\infty$ is nested, then 
$\mc{G}_{A_k}\subset\mc{G}_{A_{k+1}}$ for all $k\in\N$.
\end{itemize}
\end{lemma}

\begin{proof}
a) Fix $k\in\N$. 
Since $\{0\}\in\mc{K}_c(\R^d)$, we have
\[\inf_{C\in\mc{G}_{A_k}}\dist_H(\{0\},C)\le\alpha_k\cdot 0=0.\]
In particular, there exists $C\in\mc{G}_{A_k}$ with $\|C\|_2\le 1$,
and hence, by Theorem \ref{alternative}, the entire space $\mc{G}_{A_k}$ 
consist of polytopes.

Statement b) is trivial.
\end{proof}

Let us check that the concept of Galerkin sequences makes sense.

\begin{theorem} \label{ngsexists}
For every $d\ge 2$, there exists a nested Galerkin sequence
$(A_k)_{k=0}^\infty$ of matrices $A_k\in\R^{N_k\times d}$.
\end{theorem}

\begin{proof}
Consider spherical coordinates 
$\zeta:[0,\pi]^{d-2}\times[0,2\pi]\to S^{d-1}$ given by
$\zeta_i(\theta)=\cos(\theta_i)\prod_{j=1}^{i-1}\sin(\theta_j)$
for $i\in\{1,\ldots,d-1\}$ and 
$\zeta_d(\theta)=\prod_{j=1}^{d-1}\sin(\theta_j)$.
For every $k\in\N_1$, we consider the grid 
\[\Delta_k:=\frac{\pi}{2^k}\Z^{d-1}\cap([0,\pi]^{d-2}\times[0,2\pi]),\]
we define $\{a_1^k,\ldots,a_{N_k}^k\}:=\zeta(\Delta_k)$, and we let $A_k$ 
be the matrix consisting of the rows $(a_1^k)^T,\ldots,(a_{N_k}^k)^T$.
Since the grids $(\Delta_k)_{k\in\N}$ are nested, so are the 
matrices $(A_k)_{k=0}^\infty$.

For every $c\in S^{d-1}$, there exists $\theta\in[0,\pi]^{d-2}\times[0,2\pi]$
with $\zeta(\theta)=c$.
By definition, there exists $\tilde\theta\in\Delta_k$ with 
$\|\theta-\tilde\theta\|_\infty\le 2^{-k-1}\pi$.
An elementary computation shows 
\[\|\zeta(\theta)-\zeta(\tilde\theta)\|_2
\le\sqrt{d}\|\zeta(\theta)-\zeta(\tilde\theta)\|_\infty
\le d\sqrt{d}\|\theta-\tilde\theta\|_\infty
\le 2^{-k-1}\pi d\sqrt{d},\]
so by Theorem \ref{oldapprox}, the spaces $\mc{G}_{A_k}$ have the desired
approximation properties for sufficiently large $k$. 
\end{proof}

The following proposition reveals additional details of the relationship between two
polytope spaces from a nested Galerkin sequence, which may be of some interest
for numerical computations with adaptive refinement.

\begin{proposition} \label{nested}
Let $N_1,N_2\in\N$ with $N_1<N_2$, let $a_1,\ldots,a_{N_2}\in S^{d-1}$ 
be pairwise distinct, and let $A_1\in\R^{N_1\times d}$ and $A_2\in\R^{N_2\times d}$ 
be the matrices consisting of the rows $a_1^T,\ldots,a_{N_1}^T$ 
and $a_1^T,\ldots,a_{N_2}^T$, respectively. 
If the space $\mc{G}_{A_1}$ consists of polytopes, then the following statements hold:
\begin{itemize}
\item [a)] The space $\mc{G}_{A_2}$ consists of polytopes.
\item [b)] We have $P_{\mc{G}_{A_2}}(C)\subset P_{\mc{G}_{A_1}}(C)$ 
for any $C\in\mc{K}_c(\R^d)$.
\item [c)] We have $P_{\mc{C}_{A_2}}(\mc{G}_{A_1})
\subset\mathrm{bd}(\mc{C}_{A_2})$.
\end{itemize}
\end{proposition}

\begin{proof}
Statement a) follows from Theorem \ref{alternative} and the fact that $p\in Q_{A_1,c}^*$
implies $(p^T,0_{\R^{N_2-N_1}})^T\in Q_{A_2,c}^*$. 
Statement b) is obvious.

Let $b^1\in\mc{C}_{A_1}$, and let $b^2:=P_{\mc{C}_{A_2}}(Q_{A_1,b^1})$.
By the definitions of $\mc{C}_A$ and $P_{\mc{C}_A}$, we have $b^2_i=b^1_i$ 
for all $i\in\{1,\ldots,N_1\}$, and for every $i\in\{N_1+1,\ldots,N_2\}$, 
Theorem \ref{strong:duality} gives
\[b^2_i
=\max\{a_i^Tx:x\in Q_{A_1,b^1}\}
=\min\{(b^1)^Tp:p\in Q_{A_1,a_i}^*\}.\]
In particular, we have $Q_{A_1,a_i}^*\neq\emptyset$, and by Lemma \ref{has:vertices},
we have $\ext(Q_{A_1,a_i}^*)\neq\emptyset$, so there exists $p\in\ext(Q_{A_1,a_i}^*)$
with $b^2_i=p^Tb^1$.
By Lemma \ref{vertices}b, we have 
$(p^T,\mymathbb{0}_{N_2-N_1}^T)^T\in\ext(Q_{A_2,a_i}^*)$.
Since $(p^T,\mymathbb{0}_{N_2-N_1}^T)^T\neq e_i$ and 
\[(p^T,\mymathbb{0}_{N_2-N_1}^T)b^2=p^Tb^1=b_i^2,\] 
it follows from Theorem \ref{interior:bigger} that $b^2\notin\interior(\mc{C}_{A_2})$.
\end{proof}

Property d) above may be undesirable.
In particular, interior point methods require an initial guess in 
$\interior(\mc{C}_{A_k})$. 
A simple solution to this problem is provided in the following proposition.

\begin{proposition}
Let $A\in\R^{N\times d}$ such that $\mc{G}_A$ consists of polytopes,
and let $\lambda\in(0,1)$. 
Then the mappings
\begin{align*}
&P_{\mc{C}_A}^\lambda:\mc{K}_c(\R^d)\to\interior(\mc{C}_A),\quad 
P_{\mc{C}_A}^\lambda(C):=(1-\lambda)P_{\mc{C}_A}(C)+\lambda\|C\|_2\mathbbm{1},\\
&P_{\mc{G}_A}^\lambda:\mc{K}_c(\R^d)\to\mc{G}_A,\quad 
P_{\mc{G}_A}(C):=\phi(P_{\mc{C}_A}^\lambda(C))
\end{align*}
with $\phi$ as in Theorem \ref{phi} satisfy  
\begin{align*}
&\|P_{\mc{C}_A}^\lambda(C)-P_{\mc{C}_A}(C)\|_\infty\le 2\lambda\|C\|_2
\quad\forall\,C\in\mc{K}_c(\R^d),\\
&\dist_H(P_{\mc{G}_A}^\lambda(C),P_{\mc{G}_A}(C))\le 2\lambda L_A\|C\|_2
\quad\forall\,C\in\mc{K}_c(\R^d).
\end{align*}
\end{proposition}

\begin{proof}
According to Theorem \ref{interior:bigger}, we have $\mathbbm{1}\in\interior(\mc{C}_A)$,
and by Corollary \ref{cacc}, the coordinate space $\mc{C}_A$ is a convex cone.
Since $P_{\mc{C}_A}(C)\in\mc{C}_A$, it follows that
\[(1-\lambda)P_{\mc{C}_A}(C)+\lambda\|C\|_2\mathbbm{1}\in\interior(\mc{C}_A).\]
The estimates follow from Lemma \ref{support:Lipschitz}, the computation
\[\|P_{\mc{C}_A}^\lambda(C)-P_{\mc{C}_A}(C)\|_\infty
\le\lambda\|P_{\mc{C}_A}(C)-\|C\|_2\mathbbm{1}\|_\infty
\le\lambda(\|\sigma_C\|_\infty+\|C\|_2)
\le 2\lambda\|C\|_2,\]
and the fact that $\phi$ is $L_A$-Lipschitz.
\end{proof}

\subsection{A concrete optimization problem in $\mc{K}_c(\R^d)$} \label{concrete}

Throughout this section, we fix a continuous functional $\Phi:\mc{K}_c(\R^d)\to\R$, 
an $L$-Lipschitz constraint $\Psi:(\mc{K}_c(\R^d),\dist_H)\to(\R^m,\|\cdot\|_\infty)$ 
as well as sets $\check{C},\hat{C}\in\mc{K}_c(\R^d)$, 
and we consider the model problem
\begin{equation} \label{aop}
\min_{C\in\mc{K}_c(\R^d)}\Phi(C)\quad\text{subject to}\quad
\Psi(C)\le 0,\ \check{C}\subset C\subset\hat{C}.
\end{equation}
We fix a nested Galerkin sequence $(A_k)_{k=0}^\infty$ with $A\in\R^{N_k\times d}$
and approximate this problem with a sequence 
\begin{equation} \label{fdop}
\min_{C\in\mc{G}_{A_k}}\Phi_k(C)\quad\text{subject to}\quad
\Psi_k(C)\le 0,\ P_{\mc{G}_{A_k}}(\check{C})\subset C\subset P_{\mc{G}_{A_k}}(\hat{C})
\end{equation}
of finite-dimensional problems with suitable mappings $\Psi_k$, which become
\begin{equation} \label{cop}
\min_{b\in\mc{C}_{A_k}}\Phi_k(Q_{A_k,b})\quad\text{subject to}\quad
\Psi_k(Q_{A_k,b})\le 0,\ 
P_{\mc{C}_{A_k}}(\check{C})\le b\le P_{\mc{C}_{A_k}}(\hat{C})
\end{equation}
when expressed in coordinates.
By Corollary \ref{matrix:representation}, the constraint $b\in\mc{C}_{A_k}$
can be represented as a linear inequality.
Let us denote
\begin{align*}
&\mc{M}:=\{C\in\mc{K}_c(\R^d): \Psi(C)\le 0,\ \check{C}\subset C\subset\hat{C}\},\\
&\mc{M}_k:=\{C\in\mc{G}_{A_k}: \Psi_k(C)\le 0,\ 
P_{\mc{G}_{A_k}}(\check{C})\subset C\subset P_{\mc{G}_{A_k}}(\hat{C})\}.
\end{align*}

Conditions \eqref{nonempty} are redundant in the characterization of $\mc{M}_k$,
which is convenient from a computational perspective.

\begin{lemma} \label{dispense}
For any $b\in\R^{N_k}$ with $P_{\mc{G}_{A_k}}(\check{C})\subset Q_{A_k,b}$ and any 
$p\in\ext(Q_{A,0}^\diamond)$,  we have $0\le b^Tp$.
\end{lemma}

\begin{proof}
Since $\emptyset\neq P_{\mc{G}_{A_k}}(\check{C})\subset Q_{A_k,b}$
this follows from Proposition \ref{Farkas}a.
\end{proof}

The constraints $\Psi_k$ can be designed in such a way that 
the sets $\mc{M}_k$ converge to $\mc{M}$.
Recall the definition of the constant $\kappa_A$ from Section \ref{approxsec}.

\begin{proposition} \label{setconvprop}
The set $\mc{M}$ is compact.
Assume that $\mc{M}\neq\emptyset$, and define 
\begin{equation} \label{psidef}
\Psi_k:\mc{K}_c(\R^d)\to\R^m,\quad
\Psi_k(C):=\Psi(C)-L\kappa_{A_k}\|\hat{C}\|_2\mathbbm{1}_{\R^m}.
\end{equation}
Then the sets $\mc{M}_k$ are nonempty and compact for all $k\in\N$.
If, in addition, we have $\lim_{k\to\infty}\kappa_{A_k}=0$, then we have 
\[\mc{D}_H(\mc{M},\mc{M}_k)\to 0\quad\text{as}\quad k\to\infty.\]
\end{proposition}

\begin{proof}
By Theorem \ref{Blaschke}, the set $\{C\in\mc{K}_c(\R^d):\check{C}\subset C\subset\hat C\}$
is relatively compact.
Since $\check{C}\subset C\subset\hat C$ holds if and only if we have 
$\dist(\check{C},C)=0$ and $\dist(C,\hat C)=0$, 
and since $C\mapsto\dist(\check{C},C)$ 
and $C\mapsto\dist(C,\hat C)$ are continuous, 
the set $\{C\in\mc{K}_c(\R^d):\check{C}\subset C\subset\hat C\}$ is closed.
By continuity of $\Psi$, the set
$\{C\in\mc{K}_c(\R^d): \Psi(C)\le 0\}$ is closed as well.
All in all, the set $\mc{M}$ is compact, and the sets $\mc{M}_k$, $k\in\N$, 
are compact for the same reasons.

\medskip

Let $\mc{M}\neq\emptyset$.
All $C\in\mc{M}$ satisfy $C\in\mc{K}_c(\R^d)$ and $\Psi(C)\le 0$, 
as well as $\check{C}\subset C\subset\hat{C}$.
By Proposition \ref{projectorprop}, we have 
$P_{\mc{G}_{A_k}}(\check{C})\subset P_{\mc{G}_{A_k}}(C)\subset P_{\mc{G}_{A_k}}(\hat{C})$,
and according to Theorem \ref{projector}, we have 
\begin{equation*} \label{kommetihrhirten}
\dist_H(C,P_{\mc{G}_{A_k}}(C))\le\kappa_{A_k}\|C\|_2\le\kappa_{A_k}\|\hat C\|_2.
\end{equation*}
But then
\begin{align*}
&\Psi_k(P_{\mc{G}_{A_k}}(C))
=\Psi(P_{\mc{G}_{A_k}}(C))-L\kappa_{A_k}\|\hat{C}\|_2\mathbbm{1}_{\R^m}\\
&\le\Psi(C)+\|\Psi(P_{\mc{G}_{A_k}}(C))-\Psi(C)\|_\infty\mathbbm{1}_{\R^m}
-L\kappa_{A_k}\|\hat{C}\|_2\mathbbm{1}_{\R^m}\le 0
\end{align*}
shows $P_{\mc{G}_{A_k}}(C)\in\mc{M}_k$. 
Hence $\mc{M}_k\neq\emptyset$ and $\mc{D}(\mc{M},\mc{M}_k)\to 0$
as $k\to\infty$.

\medskip

Assume that $\mc{D}(\mc{M}_k,\mc{M})\not\to 0$ as $k\to\infty$.
Then there exist a number $\eps>0$, a subsequence $\N'\subset\N$ 
and $(C_k)_{k\in\N'}$ with $C_k\in\mc{M}_k$ for all $k\in\N'$ and 
\begin{equation} \label{contra}
\mc{D}(C_k,\mc{M})\ge\eps\quad\forall\,k\in\N'.
\end{equation}
By Lemma \ref{nested}b, we have 
$C_k\subset P_{\mc{G}_{A_k}}(\hat C)\subset P_{\mc{G}_{A_0}}(\hat C)$,
and by Lemma \ref{immediate}a, the set $P_{\mc{G}_{A_0}}(\hat C)$ bounded,
so according to Theorem \ref{Blaschke}, 
there exist a subsequence $\N''\subset\N$ 
and $C\in\mc{K}_c(\R^d)$ such that $\lim_{\N''\ni k\to\infty}\dist_H(C,C_k)=0$.
By continuity of $\Psi$ and by the definition of $\Psi_k$, we have 
\[\Psi(C)=\lim_{\N''\ni k\to\infty}\Psi(C_k)
=\lim_{\N''\ni k\to\infty}\big(\Psi_k(C_k)+L\kappa_{A_k}\|\hat{C}\|_2\mathbbm{1}_{\R^m}\big)=0.\]
Moreover, since $C_k\subset P_{\mc{G}_{A_k}}(\hat{C})$ and by Proposition
\ref{projector}, we have
\begin{align*}\dist(C,\hat{C})
\le&\dist(C,C_k)
+\dist(C_k, P_{\mc{G}_{A_k}}(\hat{C}))\\
&+\dist(P_{\mc{G}_{A_k}}(\hat{C}),\hat{C})
\to 0\quad\text{as}\quad\N''\ni k\to\infty,
\end{align*}
so $C\subset\hat{C}$.
But then $C\in\mc{M}$, which contradicts statement \eqref{contra}.
All in all, we proved that $\mc{D}(\mc{M}_k,\mc{M})\to 0$ as $k\to\infty$.
\end{proof}

Now we gather the results from this section in a final statement.

\begin{theorem} \label{concreteconv1}
For $k\in\N$, let $\Phi_k:\mc{K}_c(\R^d)\to\R$ be lower semicontinuous mappings 
which satisfy condition \eqref{uniapprox},
let constraints $\Psi_k:\mc{K}_c(\R^d)\to\R^m$ be defined by \eqref{psidef},
and assume that $\mc{M}\neq\emptyset$.
Then $\argmin_{C\in\mc{M}}\Phi(C)\neq\emptyset$, and
\[\lim_{k\to\infty}\mc{D}(\argmin_{C\in\mc{M}_k}\Phi_k(C),
\argmin_{C\in\mc{M}}\Phi(C))=0.\]
\end{theorem}

\begin{proof}
By Proposition \ref{setconvprop}, the set $\mc{M}$ is compact,
and for every $k\in\N$, there exists $C_k\in\mc{M}_k$,
and the set $\mc{M}_k$ is compact.
As $\Phi_k$ is lower semicontinuous, the nonempty sets $S(\Phi_k,\Phi_k(C_k))$ are 
closed by Lemma \ref{exargmin}a.
Lemma \ref{exargmin}c yields that $\argmin_{C\in\mc{M}_k}\Phi_k(C)\neq\emptyset$.
In addition, Proposition \ref{setconvprop} ensures that condition
\eqref{setconv} is satisfied.
Thus Theorem \ref{minconvthm} applies and yields the desired statement.
\end{proof}

\section{Conclusion}

This paper lays the foundations for a systematic numerical treatment of optimization
problems in the space $\mc{K}_c(\R^d)$. 
First applications presented in \cite{Ernst} and \cite{Harrach:19} support the 
usefulness of this approach.
On the other hand, many questions remain open.
We present them as clusters of interconnected problems.

\medskip

\emph{Cluster 1: The nature of the mapping $\phi:\mc{C}_A\to\mc{G}_A$.}\\ 
Is $\phi$ piecewise affine linear with respect to Minkowski addition? 
What is its exact local modulus of continuity?
Which local and global properties of a functional $\Phi:\mc{K}_c\to\R$ 
does the composition $\Phi\circ\phi:\mc{C}_A\to\R$ inherit?
What is the structure of $\mc{G}_A$ as a subspace of $\mc{K}_c(\R^d)$?

\medskip

\emph{Cluster 2:  More on the coordinate space $\mc{C}_A$.}\\
What are the extremal rays of $\mc{C}_A$, and does this knowledge have any implications
for practical computations?
Are there more redundancies in the full system (\ref{nonempty},\,\ref{touching})
than those identified in Theorem \ref{redundancy}?
How is the algebraic structure of the vectors $a_1,\ldots,a_N$ 
reflected by the algebraic structure of the sets $\ext(Q_{A,a_i}^*)$,
and what does this tell us about $\mc{C}_A$, $\mc{G}_A$, redundancies, etc?

\medskip

\emph{Cluster 3: The design of the matrix $A$.}\\
Is there a principle that helps designing $A$ in such a way that $\delta_A$
or $\kappa_A$ is (almost) minimized over $\R^{N\times d}$?
Can these special matrices be organized in a nested Galerkin sequence?
How to balance approximation properties of $\mc{C}_A$ with local Lipschitz properties 
of $\phi$?

\medskip

\emph{Cluster 4: Offline computations.}\\
How much can we infer about the structure of $Q_{A,b}$ for a particular $b$ 
from offline computations?
Can we speed up the enumeration of its vertices using offline computations?
Can we use offline computations to solve linear programs over $Q_{A,b}$ quickly?
Can we update the vertices of $Q_{A,b}$ efficiently under changes of $b$
using information compiled in offline computations?

\medskip

\emph{Cluster 5: Local minima.}\\
Under which conditions do local minimizers of the auxiliary problems 
\eqref{cop} converge to local minimizers of the original problem \eqref{aop}?
What about KKT points?

\medskip

Further interesting questions are whether our approach can help answer theoretical
questions about optimization problems in the space of convex bodies, and whether
the approach can be extended to spaces of nonconvex sets in a meaningful way.

\medskip

We hope that at least some of these questions will be answered in the future,
and that other researchers find this programme sufficiently interesting to contribute
to its development.

\bibliographystyle{plain}
\bibliography{galerkin}

\end{document}